\documentclass[11pt,reqno]{amsart}

\usepackage[T1]{fontenc}
\usepackage{lmodern}
\usepackage{microtype}
\usepackage{amsmath,amssymb,amsthm,mathtools}
\usepackage[colorlinks=true,citecolor=blue,linkcolor=blue,urlcolor=blue]{hyperref}

\newtheorem{theorem}{Theorem}[section]
\newtheorem{proposition}[theorem]{Proposition}
\newtheorem{lemma}[theorem]{Lemma}

\theoremstyle{definition}

\newtheorem{question}[theorem]{Question}
\theoremstyle{remark}
\newtheorem{remark}[theorem]{Remark}

\newcommand{\N}{\mathbb N}
\newcommand{\Zp}{\mathbb Z_+}

\title{$\Delta^*$-Mixing and Products of Full-Scattering Systems}

\author[H.~Xu]{Hui Xu}
\address[H.~Xu]{Department of Mathematics, Shanghai Normal University, Shanghai, 200234, CHINA}
\email{huixu@shnu.edu.cn}


\begin{document}

\begin{abstract}
We resolve two questions posed by Huang and Ye concerning full-scattering
systems.  First, we construct a four-symbol subshift which is $\Delta^*$-mixing
but not full-scattering, thereby showing that $\Delta^*$-mixing does not imply
full-scattering.  Second, we prove that every finite product of full-scattering
systems is full-scattering.
\end{abstract}

\maketitle

\section{Introduction}

Throughout, a topological dynamical system $(X,T)$ consists of a compact
metric space $X$ and a continuous surjection $T\colon X\to X$.  If
$\mathcal U$ is a finite cover of $X$, let $N(\mathcal U)$ denote the least
cardinality of a subcover.  We use the convention
$\Zp=\{0,1,2,\ldots\}$.  For an infinite set
$A=\{a_1<a_2<\cdots\}\subset\Zp$, define
\[
 c_A(\mathcal U,n)
 =N\left(\bigvee_{i=1}^nT^{-a_i}\mathcal U\right).
\]
A finite open cover $\mathcal U=\{U_1,\ldots,U_k\}$ is \emph{non-trivial}
if no $U_i$ is dense in $X$.  The system $(X,T)$ is \emph{full-scattering}
if
\[
 c_A(\mathcal U,n)\longrightarrow\infty
\]
for every infinite $A\subset\Zp$ and every non-trivial finite open cover
$\mathcal U$.

The use of open-cover complexity to distinguish topological mixing properties
goes back to Blanchard, Host and Maass \cite{BHM}.  They introduced scattering
and $2$-scattering and related these notions to weak disjointness.  Huang and
Ye subsequently characterized scattering and strong scattering through weak
disjointness and constructed an explicit scattering system which is not
weakly mixing \cite{HY02}.  In \cite{HY}, they developed the sequence-dependent
notions of $S$-scattering and $\mathcal F$-scattering and proved, among other
things, that $2$-scattering is equivalent to scattering
\cite[Theorem~2.1]{HY}.

This work places full-scattering between two familiar mixing conditions:
strong mixing implies full-scattering
\cite[Theorem~5.6(1) and Lemma~5.7(1)]{HY}, while full-scattering implies mild
mixing and hence weak mixing \cite[Theorem~2.3]{HY}.  Neither
implication is reversible for general systems
\cite[Examples~A and B and Theorem~6.6]{HY}.
Under minimality, however, full-scattering is equivalent to strong mixing
\cite[Theorem~5.8]{HY}.  Thus full-scattering is a genuinely nonminimal
open-cover complexity condition, and its permanence properties do not follow
from the corresponding facts for the usual mixing classes.

Two related notions appear in \cite{HY}.  Let $\Delta$ be the family of
subsets of $\Zp$ containing all positive differences of some infinite subset
of $\Zp$, and let $\Delta^*$ be its dual family.  Thus $E\in\Delta^*$ if and
only if $E$ meets every member of $\Delta$.  For nonempty open sets $U,V\subset
X$, write
\[
 N(U,V)=\{n\in\Zp:U\cap T^{-n}V\neq\varnothing\}.
\]
A system is $\Delta^*$-transitive if $N(U,V)\in\Delta^*$ for all such $U,V$,
and it is $\Delta^*$-mixing if its Cartesian square is
$\Delta^*$-transitive.  Corollary~6.5 of \cite{HY} shows that full-scattering
implies $\Delta^*$-transitivity, while Theorem~2.3 gives weak mixing.  Every
$\Delta^*$-set meets every $\Delta$-set infinitely often: otherwise one could
thin the infinite set generating the latter difference set so that all its
differences avoid the finite intersection.  Thus $\Delta^*$ is a full family,
and Proposition~3.6 of \cite{HY} gives $\Delta^*$-mixing.  The converse was
left open and was still recorded as open in subsequent work on
Furstenberg-family mixing \cite[p.~9]{DA}.  Also, if
$A=\{a_1<a_2<\cdots\}$, put
\[
 h_A(T,\mathcal U)
 =\limsup_{n\to\infty}\frac1n
  \log N\left(\bigvee_{i=1}^nT^{-a_i}\mathcal U\right).
\]
A system is \emph{full-Bernoulli} if $h_A(T,\mathcal U)>0$ for every
infinite $A\subset\Zp$ and every non-trivial finite open cover $\mathcal U$.
This property implies full-scattering; systems with Blanchard's strong
property $P$ are full-Bernoulli \cite[Lemma~6.3]{HY}, and Blanchard's example
\cite{Blanchard} has strong property $P$ but is not strongly mixing
\cite[Example~A]{HY}.
Huang, Shao and Ye subsequently proved that full-Bernoulli systems are
precisely the topological $K$-systems, namely, the systems for which every
non-trivial finite open cover has positive topological entropy
\cite[Theorem~5.5 and Corollary~5.9]{HSY}.

Huang and Ye \cite[Question~1]{HY} posed the following three-part problem.

\begin{question}[Huang--Ye]\label{ques:HY}
\mbox{}\par
\begin{enumerate}
\item Does $\Delta^*$-mixing imply full-scattering?
\item Is the product of two full-scattering systems still full-scattering?
\item Is the product of two full-Bernoulli systems still full-Bernoulli?
\end{enumerate}
\end{question}

The third part of Question~\ref{ques:HY} was answered affirmatively in
\cite[Corollary~5.9]{HSY}, since the class of topological $K$-systems is
closed under products.  To the best of our knowledge, the first two parts
have remained open.  We answer the first negatively and the second
affirmatively.

\begin{theorem}\label{thm:counterexample}
There exists a $\Delta^*$-mixing system which is not full-scattering.
\end{theorem}

\begin{theorem}\label{thm:main}
Every finite product of full-scattering systems is full-scattering.
\end{theorem}

The example in Theorem~\ref{thm:counterexample} is a four-symbol subshift.
It is defined by forbidding the simultaneous occurrence of all four symbols
on any translate of the sparse set $\{4^m:m\geq1\}$.  This immediately gives
bounded cover complexity along that set.  On the other hand, finite-support
points are dense, and two such points can be placed far apart except at a set
of displacements contained in finitely many translates of $\{4^m:m\geq1\}$.
Such a finite union contains no $\Delta$-set, which yields
$\Delta^*$-mixing.

\section{A \texorpdfstring{$\Delta^*$}{Delta-star}-mixing system which is not
full-scattering}

We begin with the elementary sparse-set fact used in the construction.  For
$D\subset\Zp$, put
\[
 \Delta(D)=\{d'-d:d,d'\in D,\ d<d'\}.
\]

\begin{lemma}\label{lem:sparse}
Let $B=\{4^m:m\geq1\}$.  If $c_1,\ldots,c_s\in\mathbb Z$ and $F\subset\Zp$
is finite, then
\[
 F\cup\bigcup_{j=1}^s(c_j+B)
\]
does not contain $\Delta(D)$ for any infinite $D\subset\Zp$.
\end{lemma}

\begin{proof}
Suppose otherwise.  By passing recursively to an infinite subset of $D$, we
may assume that all of its positive differences exceed $\max F$ (with this
step omitted when $F$ is empty).  Thus $\Delta(D)\cap F=\varnothing$.  Color each pair
$d<d'$ in $D$ by an index $j$ for which $d'-d\in c_j+B$.  The infinite
Ramsey theorem \cite[Section~9.1]{Diestel} gives an infinite $D'\subset D$
and a fixed $c\in\{c_1,\ldots,c_s\}$
such that $\Delta(D')\subset c+B$.

Choose $d_1<d_2<d_3$ in $D'$.  For some $p,q,r\geq1$,
\[
 d_2-d_1=c+4^p,\qquad d_3-d_2=c+4^q,\qquad
 d_3-d_1=c+4^r.
\]
Consequently
\begin{equation}\label{eq:powers}
 4^r=4^p+4^q+c.
\end{equation}
The third difference is larger than each of the first two, so $r>p,q$.
Moreover, since $D'$ is infinite, $d_3-d_1$ and hence $r$ can be chosen
arbitrarily large.  Therefore
\[
 4^r=4^p+4^q+c
 \leq 2\cdot4^{r-1}+|c|<4^r,
\]
a contradiction.
\end{proof}

We use the two-sided full shift $\{0,1,2,3\}^{\mathbb Z}$ with the left shift
$\sigma$.  For $x\in\{0,1,2,3\}^{\mathbb Z}$, write
\[\operatorname{supp}(x)=\{n\in\mathbb Z:x_n\neq0\}.\]
  Let $Y$ be the set of
all finite-support points $x$ such that
\begin{equation}\label{eq:forbidden}
 \{x_{k+b}:b\in B\}\neq\{0,1,2,3\}
 \qquad\text{for every }k\in\mathbb Z,
\end{equation}
and set $X=\overline{Y}$.  Condition \eqref{eq:forbidden} is invariant under
$\sigma$ and $\sigma^{-1}$, so $\sigma(Y)=Y$.
Let $Z$ be the set of all points of the full shift satisfying
\eqref{eq:forbidden}.  If $x\notin Z$, then for some $k\in\mathbb Z$ and
$b_0,b_1,b_2,b_3\in B$ one has
\[
 x_{k+b_i}=i\qquad(0\leq i\leq3).
\]
The cylinder determined by these four equalities is contained in
$\{0,1,2,3\}^{\mathbb Z}\setminus Z$.  Hence $Z$ is closed.  Since
$Y\subset Z$, the space $X=\overline Y$ is a compact shift-invariant subspace
of $Z$, and $\sigma|_X$ is a homeomorphism.

\begin{lemma}\label{lem:gluing}
Let $(U_j,V_j)$, $1\leq j\leq s$, be finitely many pairs of nonempty open
subsets of $X$.  There is a set $E\in\Delta^*$ such that
\[
 E\subset\bigcap_{j=1}^sN(U_j,V_j).
\]
\end{lemma}

\begin{proof}
Since $Y$ is dense in $X$, choose $x^j\in U_j\cap Y$ and
$y^j\in V_j\cap Y$.  Set
\[
 A_j=\operatorname{supp}(x^j),\qquad
 D_j=\operatorname{supp}(y^j).
\]
Choose finite sets $L_j,M_j\subset\mathbb Z$ such that the relative
cylinders
\[
 [x^j]_{L_j}=\{u\in X:u_t=x^j_t\text{ for every }t\in L_j\}\subset U_j,
 \qquad
 [y^j]_{M_j}\subset V_j.
\]
Let $F_j^0$ be the set of $n\in\Zp$ for which at least one of the three sets
in
\begin{equation}\label{eq:exceptional}
 A_j\cap(n+D_j),\qquad L_j\cap(n+D_j),\qquad
 A_j\cap(n+M_j)
\end{equation}
is nonempty.  Since $A_j,D_j,L_j,M_j$ are finite, $F_j^0$ is finite.

For $n\in\Zp\setminus F_j^0$, define $z^{j,n}$ by
\begin{equation}\label{eq:z-definition}
 z^{j,n}_t=
 \begin{cases}
  x^j_t,&t\in A_j,\\
  y^j_{t-n},&t\in n+D_j,\\
  0,&t\notin A_j\cup(n+D_j).
\end{cases}
\end{equation}
Equivalently, $z^{j,n}$ consists of the nonzero block of $x^j$ on $A_j$
and the translated nonzero block of $y^j$ on $n+D_j$, with value $0$ at
all remaining coordinates.
The first set in \eqref{eq:exceptional} ensures that the first two cases in
\eqref{eq:z-definition} are disjoint.  The other two sets give
\begin{equation}\label{eq:cylinder-agreement}
 z^{j,n}|_{L_j}=x^j|_{L_j},\qquad
 (\sigma^nz^{j,n})|_{M_j}=y^j|_{M_j}.
\end{equation}
Consequently, if $z^{j,n}$ satisfies \eqref{eq:forbidden}, then
$z^{j,n}\in Y$ and
\begin{equation}\label{eq:return-from-z}
 z^{j,n}\in U_j,\qquad \sigma^nz^{j,n}\in V_j.
\end{equation}

Suppose that $z^{j,n}$ does not satisfy \eqref{eq:forbidden}.  There is then
$k\in\mathbb Z$ such that each of the symbols $0,1,2,3$ occurs on $k+B$.
Since $z^{j,n}$ has finite support and $k+B$ is infinite, the symbol $0$
occurs on $k+B$.  Hence there are $r_1,r_2,r_3\in k+B$ such that
\begin{equation}\label{eq:three-colors}
 z^{j,n}_{r_i}=i\qquad(1\leq i\leq3).
\end{equation}
The positions in \eqref{eq:three-colors} cannot all belong to $A_j$, since
$x^j$ satisfies \eqref{eq:forbidden}.  They cannot all belong to $n+D_j$
either, since the translated copy of $y^j$ also satisfies
\eqref{eq:forbidden}.  Thus two positions belong to one of $A_j,n+D_j$ and
the remaining position belongs to the other.

We use the following uniqueness property:
\begin{equation}\label{eq:unique-difference}
 4^a-4^b=4^{a'}-4^{b'}>0,\quad a>b,\quad a'>b'
 \quad\Longrightarrow\quad (a,b)=(a',b').
\end{equation}
Indeed, the largest power of $4$ dividing the common difference determines
$b$, after which $4^{a-b}-1$ determines $a-b$.

Assume first that two positions in \eqref{eq:three-colors} are
$u,v\in A_j$, where $u>v$.  Since $u,v\in k+B$, property
\eqref{eq:unique-difference} determines $\alpha>\beta$ from
\[
 u-v=4^\alpha-4^\beta
\]
and then determines $k=u-4^\alpha$.  Write the third position as $n+w$,
where $w\in D_j$.  For some $\gamma\geq1$,
\[
 n+w=k+4^\gamma,
\]
so $n=(k-w)+4^\gamma\in c+B$, with $c=k-w$.  The finite sets $A_j,D_j$
give only finitely many possible constants $c$.

Assume next that two positions in \eqref{eq:three-colors} are $n+u,n+v$,
where $u,v\in D_j$ and $u>v$.  Property
\eqref{eq:unique-difference} determines $\alpha>\beta$ from
\[
 u-v=4^\alpha-4^\beta,
\]
and $k=n+u-4^\alpha$.  If the third position is $w\in A_j$, then
$w=k+4^\gamma$ for some $\gamma\geq1$, and hence
\[
 n=(w-u+4^\alpha)-4^\gamma\in c-B.
\]
For each fixed $c$, the set $(c-B)\cap\Zp$ is finite.  Thus the bad times of
this second type form a finite set.

It follows that there are a finite set $F_j\subset\Zp$ and integers
$c_{j,1},\ldots,c_{j,t_j}$ such that
\begin{equation}\label{eq:bad-times}
 \{n\in\Zp\setminus F_j^0:z^{j,n}\notin Y\}
 \subset F_j\cup\bigcup_{\nu=1}^{t_j}(c_{j,\nu}+B).
\end{equation}
Enlarge $F_j$ to contain $F_j^0$, and let $H$ be the union of the right-hand
sides of \eqref{eq:bad-times} over $1\leq j\leq s$.  By
Lemma~\ref{lem:sparse}, $H$ contains no $\Delta$-set.  Therefore
$E=\Zp\setminus H$ belongs to $\Delta^*$, and
\eqref{eq:return-from-z} gives
$E\subset\bigcap_{j=1}^sN(U_j,V_j)$.
\end{proof}

\begin{proof}[Proof of Theorem~\ref{thm:counterexample}]
Taking two pairs in Lemma~\ref{lem:gluing} shows that for all nonempty open
$U_1,U_2,V_1,V_2\subset X$,
\[
 N(U_1,V_1)\cap N(U_2,V_2)\in\Delta^*.
\]
Here we used that $\Delta^*$ is upward hereditary.  Equivalently,
$(X\times X,\sigma\times\sigma)$ is
$\Delta^*$-transitive, so $(X,\sigma)$ is $\Delta^*$-mixing.

For $i\in\{0,1,2,3\}$, let
\[
 U_i=\{x\in X:x_0\neq i\}.
\]
Each $U_i$ is clopen and is not dense: the configuration taking the value
$i$ at the origin and $0$ elsewhere belongs to $X\setminus U_i$.  Thus
$\mathcal U=\{U_0,U_1,U_2,U_3\}$ is a non-trivial open cover of $X$.
For every $x\in X$, condition \eqref{eq:forbidden} with $k=0$ says that some
symbol $i$ does not occur at the coordinates in $B$.  Therefore
\[
 X=\bigcup_{i=0}^3\bigcap_{b\in B}\sigma^{-b}U_i.
\]
Writing $B=\{b_1<b_2<\cdots\}$, we obtain
\[
 N\left(\bigvee_{j=1}^n\sigma^{-b_j}\mathcal U\right)\leq4
 \qquad(n\geq1).
\]
Hence $(X,\sigma)$ is not full-scattering.
\end{proof}

\section{Finite self-products}

We first record the orbit-graph argument used later in the proof.

\begin{lemma}\label{lem:orbit-graph}
Let $(X,T)$ be a weakly mixing system.  Given $r\geq1$, let
$\mathcal U=\{U_1,\ldots,U_m\}$ be a non-trivial finite open cover of $X^r$.
There are $k_1,\ldots,k_r\in\Zp$ such that the equivariant map
\[
 \Phi\colon X\longrightarrow X^r,
 \qquad
 \Phi(x)=(T^{k_1}x,\ldots,T^{k_r}x),
\]
has the property that $\Phi^{-1}\mathcal U$ is a non-trivial open cover of
$X$.
\end{lemma}

\begin{proof}
For every $j\in\{1,\ldots,m\}$, choose nonempty open sets
$O_{j,s}\subset X$, $1\leq s\leq r$, such that
\[
 O_{j,1}\times\cdots\times O_{j,r}
 \subset X^r\setminus\overline{U_j}.
\]
Set $k_1=0$.  Suppose that $k_1,\ldots,k_{s-1}$ have been chosen so that
\[
 V_j^{s-1}=\bigcap_{q=1}^{s-1}T^{-k_q}O_{j,q}
\]
is nonempty and open for every $j$.  Weak mixing implies transitivity of every
finite self-product.  Hence there exists $k_s\in\Zp$ such that
\[
 (T^{(m)})^{k_s}\left(\prod_{j=1}^mV_j^{s-1}\right)
 \cap\prod_{j=1}^mO_{j,s}\neq\varnothing.
\]
It follows that
$V_j^s=V_j^{s-1}\cap T^{-k_s}O_{j,s}$ is nonempty for every $j$.
Continuing inductively gives the required $k_1,\ldots,k_r$.
Indeed,
\[
 V_j^r\subset
 \Phi^{-1}(O_{j,1}\times\cdots\times O_{j,r})
 \subset X\setminus\Phi^{-1}(U_j),
\]
so no member of $\Phi^{-1}\mathcal U$ is dense.  Equivariance follows from
the definition of $\Phi$.
\end{proof}

\begin{proposition}\label{prop:self-products}
Every finite self-product of a full-scattering system is full-scattering.
\end{proposition}

\begin{proof}
Let $(X,T)$ be full-scattering.  It is weakly mixing
\cite[Theorem~2.3]{HY}.  Fix $r\geq1$, an infinite
$A=\{a_1<a_2<\cdots\}\subset\Zp$, and a non-trivial finite open cover
$\mathcal U$ of $X^r$.  Let $\Phi$ be given by
Lemma~\ref{lem:orbit-graph}.  Equivariance gives
\[
 \bigvee_{i=1}^nT^{-a_i}(\Phi^{-1}\mathcal U)
 =\Phi^{-1}\left(
 \bigvee_{i=1}^n(T^{(r)})^{-a_i}\mathcal U\right).
\]
A subcover before pullback gives a subcover after pullback.  Therefore
\[
 N\left(\bigvee_{i=1}^n(T^{(r)})^{-a_i}\mathcal U\right)
 \geq
 N\left(\bigvee_{i=1}^nT^{-a_i}(\Phi^{-1}\mathcal U)\right).
\]
The right-hand side tends to infinity because
$\Phi^{-1}\mathcal U$ is non-trivial.
\end{proof}

\section{Permutation and finite-intersection lemmas}

For $r\geq1$, let $S_r$ denote the symmetric group of all permutations of
$\{1,\ldots,r\}$.  For an infinite set $B$, let $[B]^r$ denote the family of
all $r$-element subsets of $B$.

\begin{lemma}\label{lem:transversal}
Let $(X,T)$ be full-scattering, let $U_1,\ldots,U_r$ be pairwise disjoint
nonempty open sets, where $r\geq2$, and let $B\subset\Zp$ be infinite.
Then there are distinct times $b_1,\ldots,b_r\in B$, a point $x\in X$, and
a permutation $\rho\in S_r$ such that
\[
 T^{b_j}x\in U_{\rho(j)}
 \qquad (1\leq j\leq r).
\]
\end{lemma}

\begin{proof}
Choose closed sets $F_i\subset U_i$ with nonempty interiors.  If the
conclusion failed, then for every $x\in X$ there would be an index $i$ such
that $T^b x\notin F_i$ for all $b\in B$.  Otherwise one could choose, for
each $i$, a time at which the orbit of $x$ enters $F_i$; those times would be
distinct because the $F_i$ are disjoint.  Thus
\[
 X=\bigcup_{i=1}^r\bigcap_{b\in B}T^{-b}(X\setminus F_i).
\]
The family $\mathcal A=\{X\setminus F_i:1\leq i\leq r\}$ is a non-trivial
open cover.  Writing $B=\{b_1<b_2<\cdots\}$, the $r$ constant-name atoms
\[
 \bigcap_{j=1}^nT^{-b_j}(X\setminus F_i),
 \qquad 1\leq i\leq r,
\]
cover $X$ for every $n$.  Hence
\[
 N\left(\bigvee_{j=1}^nT^{-b_j}\mathcal A\right)\leq r
\]
for all $n$, contradicting full-scattering.
\end{proof}

\begin{lemma}\label{lem:all-permutations}
Let $(X,T)$ be full-scattering, let $U_1,\ldots,U_r$ be pairwise disjoint
nonempty open sets, and let $B\subset\Zp$ be infinite.  For every infinite
$C\subset B$ there is
$F=\{b_1<\cdots<b_r\}\in[C]^r$ such that, for every $\tau\in S_r$, some
$x_\tau\in X$ satisfies
\[
 T^{b_j}x_\tau\in U_{\tau(j)}
 \qquad (1\leq j\leq r).
\]
\end{lemma}

\begin{proof}
For $r=1$, choose any $b_1\in C$ and use surjectivity of $T^{b_1}$.
Suppose $r\geq2$.  By Proposition~\ref{prop:self-products}, the product
 $X^{S_r}$, with one coordinate for each $\sigma\in S_r$, is
full-scattering.  For $1\leq j\leq r$, put
\[
 W_j=\prod_{\sigma\in S_r}U_{\sigma(j)}\subset X^{S_r}.
\]
The sets $W_1,\ldots,W_r$ are pairwise disjoint and nonempty.  Apply
Lemma~\ref{lem:transversal} in $X^{S_r}$ along $C$.  We obtain
$F=\{b_1<\cdots<b_r\}\in[C]^r$, a point
$z=(x_\sigma)_{\sigma\in S_r}$, and $\rho\in S_r$ such that
\[
 (T^{S_r})^{b_j}z\in W_{\rho(j)}
 \qquad(1\leq j\leq r).
\]
Given $\tau\in S_r$, take the coordinate indexed by
 $\sigma=\tau\circ\rho^{-1}$.  At time $b_j$, this coordinate belongs to
\[
 U_{\sigma(\rho(j))}=U_{\tau(j)},
\]
which proves the assertion.
\end{proof}

\section{Proof of the product theorem}

We first note a standard consequence of weak mixing.

\begin{lemma}\label{lem:no-isolated}
A non-singleton weakly mixing compact metric system has no isolated points.
\end{lemma}

\begin{proof}
If $x$ were isolated, then $\{(x,x)\}$ would be open in $X^2$.
For every nonempty open $W\subset X^2$, transitivity applied to
$\{(x,x)\}$ and $W$ would give an $n\in\Zp$ such that
$(T\times T)^n(x,x)\in W$.  Hence the orbit of $(x,x)$ would be dense in
$X^2$.  This orbit is contained in the closed diagonal, so $X^2$ would equal
its diagonal and $X$ would be a singleton.
\end{proof}
The following direct consequence of the infinite Ramsey theorem
\cite[Section~9.1]{Diestel} will be used in the proof of
Theorem~\ref{thm:main}.

\begin{lemma}\label{lem:ramsey}
Fix $r\geq1$ and an infinite set $B$.  Suppose
$\mathcal G_1,\ldots,\mathcal G_s\subset[B]^r$ have the property that
\[
 \mathcal G_i\cap[C]^r\neq\varnothing
\]
for every infinite $C\subset B$ and every $i$.  Then
$\bigcap_{i=1}^s\mathcal G_i$ has the same property.
\end{lemma}

\begin{proof}[Proof of Theorem~\ref{thm:main}]
It suffices to prove closure under products of two systems.  Let $(X,T)$ and
$(Y,S)$ be full-scattering, and put $Z=X\times Y$ and $R=T\times S$.
The conclusion is immediate if one of the factors is a singleton.  Otherwise
both systems are weakly mixing \cite[Theorem~2.3]{HY} and, by
Lemma~\ref{lem:no-isolated}, neither space has isolated points.

Suppose that $Z$ is not full-scattering.  Since the complexity functions are
nondecreasing, there are an infinite set
$A=\{a_1<a_2<\cdots\}\subset\Zp$, a non-trivial finite open cover
$\mathcal U=\{U_1,\ldots,U_k\}$ of $Z$, and $m\geq1$ such that
\begin{equation}\label{eq:bounded}
 N\left(\bigvee_{t=1}^nR^{-a_t}\mathcal U\right)\leq m
 \qquad(n\geq1).
\end{equation}
Let $C_i=\overline{U_i}$.  Then $\{C_1,\ldots,C_k\}$ is a closed cover of
$Z$ by proper sets, and \eqref{eq:bounded} remains valid for the corresponding
closed joins.

We claim that there are names
$\omega^1,\ldots,\omega^m\in\{1,\ldots,k\}^{\N}$ such that
\begin{equation}\label{eq:infinite-atoms}
 Z=\bigcup_{q=1}^m\bigcap_{t\geq1}R^{-a_t}C_{\omega_t^q}.
\end{equation}
Indeed, for each $n$, let $H_n$ be the set of $m$-tuples of infinite names
whose length-$n$ closed atoms cover $Z$.  By \eqref{eq:bounded}, $H_n$ is
nonempty: a cover by at most $m$ length-$n$ atoms can be padded by repeated
atoms and each finite name can be extended arbitrarily.  Membership in $H_n$
depends only on the first $n$ coordinates of the names, so $H_n$ is a closed
subset of the compact space
$(\{1,\ldots,k\}^{\N})^m$.  Moreover, $H_{n+1}\subset H_n$.  A point of
$\bigcap_nH_n$ exists by compactness.  It gives
\eqref{eq:infinite-atoms}: if $z\in Z$, then for each $n$ at least one of the
$m$ length-$n$ atoms contains $z$.  The corresponding nonempty subsets of
$\{1,\ldots,m\}$ decrease with $n$, so one index $q$ works for every $n$.

There are only finitely many vectors
$(\omega_t^1,\ldots,\omega_t^m)$.  Hence there is an infinite set of indices
$I\subset\N$ on which this vector is constant.  Put
$B=\{a_t:t\in I\}$.  If $c_1,\ldots,c_r$ are the distinct coordinates of
the constant vector, then \eqref{eq:infinite-atoms} implies
\begin{equation}\label{eq:constant-atoms}
 Z=\bigcup_{\ell=1}^rK_\ell,
 \qquad
 K_\ell=\bigcap_{b\in B}R^{-b}C_{c_\ell}.
\end{equation}
We have $r\geq2$.  Otherwise $Z=K_1$; for any $b\in B$ this gives
$R^b(Z)\subset C_{c_1}$, and surjectivity of $R^b$ contradicts the
properness of $C_{c_1}$.

For each $\ell$, choose a nonempty open rectangle
$U_\ell\times V_\ell\subset Z\setminus C_{c_\ell}$.  Since $X$ and $Y$
have no isolated points, there are pairwise disjoint nonempty open sets
\[
 U'_\ell\subset U_\ell
 \quad\text{and}\quad
 V'_\ell\subset V_\ell,
 \qquad 1\leq\ell\leq r.
\]
To see this, choose distinct points successively in the prescribed open sets
and then take sufficiently small pairwise disjoint neighborhoods.

Let $\mathcal G_X\subset[B]^r$ consist of those
$F=\{b_1<\cdots<b_r\}$ for which every permutation of
$U'_1,\ldots,U'_r$ is realized at the times $b_1,\ldots,b_r$ by some point
of $X$.  Define $\mathcal G_Y$ analogously using $V'_1,\ldots,V'_r$.
Lemma~\ref{lem:all-permutations} says that each of these families meets
$[C]^r$ for every infinite $C\subset B$.  By Lemma~\ref{lem:ramsey}, choose
\[
 F=\{b_1<\cdots<b_r\}\in\mathcal G_X\cap\mathcal G_Y.
\]
Taking the identity permutation in each factor, there are $x\in X$ and
 $y\in Y$ such that
\[
 T^{b_\ell}x\in U'_\ell
 \quad\text{and}\quad
 S^{b_\ell}y\in V'_\ell
 \qquad(1\leq\ell\leq r).
\]
Consequently,
\[
 R^{b_\ell}(x,y)\in U'_\ell\times V'_\ell
 \subset Z\setminus C_{c_\ell},
\]
so $(x,y)\notin K_\ell$ for every $\ell$.  This contradicts
\eqref{eq:constant-atoms}.  Therefore $X\times Y$ is full-scattering.
Induction completes the proof for finite products.
\end{proof}

\begin{remark}
The proof uses only finite self-products and the ordinary infinite Ramsey
theorem.  The symmetric-group self-product is the step that makes every
ordering available on a single finite time set; the Ramsey argument then
selects one such set that works in both factors.
\end{remark}


\begin{thebibliography}{99}

\bibitem{DA}
D. Ahmadi Dastjerdi and M. Dabbaghian Amiri,
\emph{Types of mixings and transitivities in topological dynamics},
arXiv:1710.01932, 2017,
\url{https://arxiv.org/abs/1710.01932}.

\bibitem{Blanchard}
F. Blanchard,
\emph{Fully positive topological entropy and topological mixing},
in \emph{Symbolic dynamics and its applications} (New Haven, CT, 1991),
Contemp. Math., vol.~135, Amer. Math. Soc., Providence, RI, 1992, pp.~95--105.

\bibitem{BHM}
F. Blanchard, B. Host, and A. Maass,
\emph{Topological complexity},
Ergodic Theory Dynam. Systems \textbf{20} (2000), no.~3, 641--662,
\url{https://doi.org/10.1017/S0143385700000341}.

\bibitem{Diestel}
R. Diestel,
\emph{Graph theory}, sixth ed., Graduate Texts in Mathematics, vol.~173,
Springer, Berlin, Heidelberg, 2025,
\url{https://doi.org/10.1007/978-3-662-70107-2}.

\bibitem{HY02}
W. Huang and X. Ye,
\emph{An explicit scattering, non-weakly mixing example and weak
disjointness},
Nonlinearity \textbf{15} (2002), no.~3, 849--862,
\url{https://doi.org/10.1088/0951-7715/15/3/320}.

\bibitem{HY}
W. Huang and X. Ye,
\emph{Topological complexity, return times and weak disjointness},
Ergodic Theory Dynam. Systems \textbf{24} (2004), no.~3, 825--846,
\url{https://doi.org/10.1017/S0143385703000543}.

\bibitem{HSY}
W. Huang, S. Shao, and X. Ye,
\emph{Mixing via sequence entropy},
in \emph{Algebraic and topological dynamics}, Contemp. Math., vol.~385,
Amer. Math. Soc., Providence, RI, 2005, pp.~101--122,
MR2180232, \url{https://doi.org/10.1090/conm/385/07193}.

\end{thebibliography}
\end{document}